\documentclass{journallet}
\usepackage{graphicx,amssymb,amsmath,amsfonts,fancyhdr,here}
\usepackage{latexsym}
\usepackage[small]{caption}

\newcommand{\NN}{\mathbb{N}} 
\newcommand{\RR}{\mathbb{R}} 

\newcommand{\eps}{\varepsilon}

\newcommand\area{{\rm area}}

\def\etal{{et~al.}}

\def\ie{{i.e.}}

\newcommand{\later}[1]{{}}
\newcommand{\old}[1]{{}}

\renewcommand{\headrulewidth}{0pt} 

\title{Packing anchored rectangles\footnote{A preliminary 
version of this paper appeared in the 
{\em Proceedings of the 23rd ACM-SIAM Symposium on Discrete Algorithms},  
(SODA 2012), Kyoto, Japan, January 2012.}}

\author{
Adrian Dumitrescu\thanks{Department of Computer Science,
University of Wisconsin--Milwaukee, USA\@.
Email: \texttt{dumitres@uwm.edu}.
Supported in part by the NSF grant DMS-1001667.}
\and
Csaba D. T\'oth\thanks{Department of Mathematics, University of
  Calgary, Canada and Department of Computer Science, Tufts
  University, Medford, MA, USA\@. 
Email: \texttt{cdtoth@ucalgary.ca}. 
Supported in part by the NSERC grant RGPIN 35586 and
the NSF grant CCF-0830734.}
}

\begin{document}
\maketitle

\begin{abstract}
Let $S$ be a set of $n$ points in the unit square $[0,1]^2$, one of
which is the origin. We construct $n$ pairwise interior-disjoint
axis-aligned empty rectangles such that the lower left corner of each
rectangle is a point in $S$, and the rectangles jointly cover at least
a positive constant area (about $0.09$). This is a first step towards
the solution of a longstanding conjecture that the rectangles in such
a packing can jointly cover an area of at least 1/2.
\end{abstract}

\section{Introduction}

We consider a rectangle packing problem proposed by 
Allen Freedman~\cite[p.~345]{Tu69} in the 1960s; see also~\cite[p.~113]{CFG91}. 
More recently, the problem was brought again to attention 
(including ours) by Peter Winkler~\cite{IBM04,Winkler07,Winkler10a,Winkler10b}.
It is a one-round game between Alice and Bob. First,
Alice chooses a finite point set $S$ in the unit square $U=[0,1]^2$ in
the plane, including the origin, that is, $(0,0)\in S$
(Fig.~\ref{fig:1}(a)). Then Bob chooses an axis-parallel rectangle
$r(s)\subseteq U$ for each point $s\in S$ such that $s$ is the
lower left corner of $r(s)$, and the interior of $r(s)$ is disjoint
from all other rectangles (Fig.~\ref{fig:1}(b)).
The rectangle $r(s)$ is said to be {\em anchored} at $s$, but $r(s)$
contains no point from $S$ in its interior.
It is conjectured that for any finite set $S\subset U$, $(0,0)\in S$,
Bob can choose such rectangles that jointly cover at least half of $U$.
However, it has not even been known whether Bob can always cover
at least a positive constant area. It is clear that Bob cannot always cover
$\frac{1}{2}+\eps$ area for any fixed $\eps>0$. If Alice
chooses $S$ to be a set of $n$ equally spaced points along the
diagonal $[(0,0),(1,1)]$, as in Fig.~\ref{fig:1}(c), then the total
area of Bob's rectangles is at most $\frac{1}{2}+\frac{1}{2n}$.
There has been no progress on this problem for more than 40 years, even
though it appeared several times in the literature.
\begin{figure}[htbp]
\centering
\includegraphics[width=.85\textwidth]{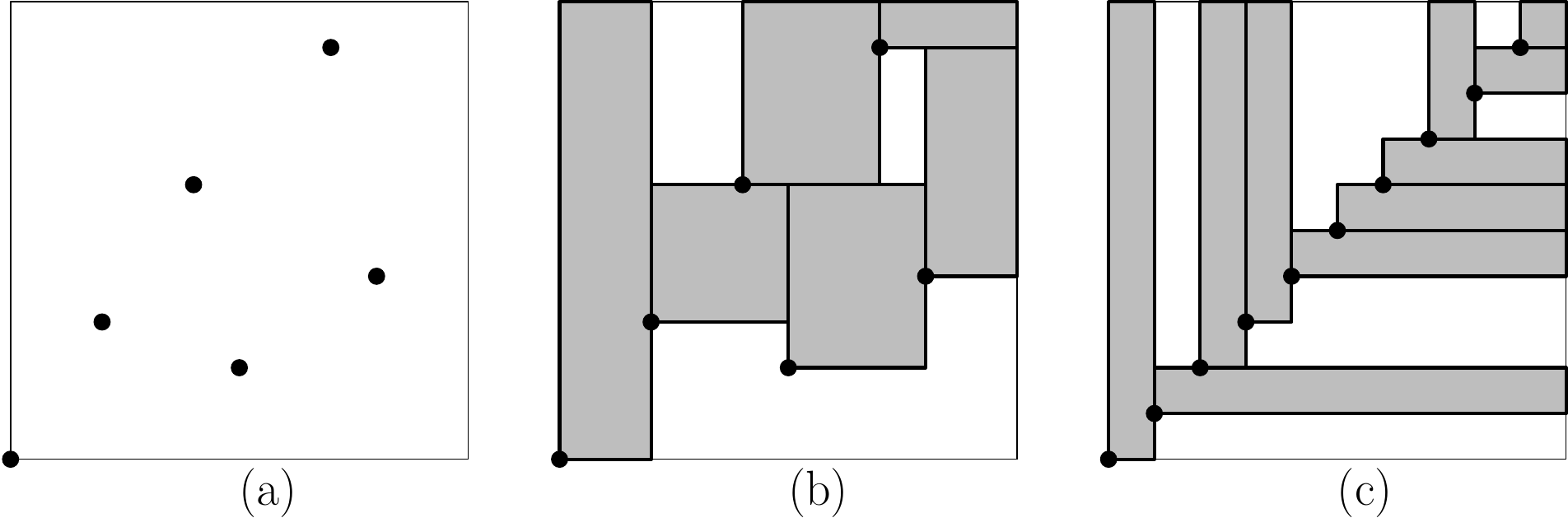}
\caption{(a) A set $S$ of 6 points in a unit square $[0,1]^2$,
  including the origin $(0,0)$. 
(b) A rectangle packing where the lower left corner of each rectangle
  is a point in $S$. 
(c) Ten equally spaced points along the diagonal $[(0,0),(1,1)]$, and
  a corresponding rectangle packing that covers roughly 1/2 area.}
\label{fig:1}
\end{figure}

\paragraph{Outline.}
In this paper, we present two simple strategies for Bob that cover at least
$0.09121$ area. These are the {\sc GreedyPacking} and the
{\sc TilePacking} algorithms described below. Both algorithms process the
points in the same specific order, namely the decreasing order of the sum of the two
coordinates, with ties broken arbitrarily (hence $(0,0)$ is the last
point processed).

The {\sc GreedyPacking} algorithm chooses a rectangle of maximum area
for each point in $S$ sequentially, in the above order.

The {\sc TilePacking} algorithm partitions $U$ into staircase-shaped tiles, 
and then chooses a rectangle of maximum area within each tile independently.
We next describe how the tiling is obtained.
Each tile is a staircase-shaped polygon, with a vertical left
side, a horizontal bottom side, and a descending staircase
connecting them. The lower left corner of each tile is
a point in $S$. We say that the tile is {\em anchored} at that point.
The algorithm maintains the invariant that the set of unprocessed
points are in the interior of a staircase shaped polygon (super-tile),
and in addition the {\em anchor} and possibly other points
are on its left and lower sides.
Processing a point amounts to shooting a horizontal ray to the right
and a vertical ray upwards which together isolate a new tile
anchored at that point, and the new staircase shaped polygon
containing the remaining points is updated.
Since $(0,0) \in S$, {\sc TilePacking} does indeed compute a tiling of
the unit square.

It will be shown shortly (Lemma~\ref{lem:greedy}) that the {\sc GreedyPacking}
algorithm  covers at least as much area as {\sc TilePacking}. Hence it suffices to
analyze the performance of the latter. The bulk of the work is in the analysis of
this simple {\sc TilePacking} algorithm, which involves geometric considerations
and a charging scheme.

\paragraph{Related work.}
Very little is known about anchored rectangle packing. Recently,
Christ~\etal~\cite{CFG+11} proved that {\em if} Alice can force Bob's
share to be less than $\frac{1}{r}$, then $n\geq 2^{2^{\Omega(r)}}$. 
Our result indicates that this condition does not materialize 
for large $r$, since Bob can always cover at least
a constant fraction of the area for any $n\in \NN$.

Previous results on rectangle packing typically consider optimization
problems, and are only loosely related to our work.
While our focus here is not on the optimization version of the
anchored rectangle packing problem,
in which the total area of the anchored rectangles
is to be maximized for a given set $S$ of anchors, our
algorithms do provide a constant-factor approximation.

In the classical {\em strip packing} problem, $n$ given axis-aligned
rectangles should be placed (without rotation or overlaps) in a
rectangular container of width 1 and minimum height. This problem is
APX-hard (by a reduction from bin packing). After a series of previous
results ({\em e.g.},~\cite{Sch94,Ste97}), Harren~\etal~\cite{HJP+11}
recently found a $(5/3+\eps)$-approximation. Jensen and
Solis-Oba~\cite{JSO08} devised an AFPTAS which packs the rectangles
into a box of height at most $(1+\eps){\rm OPT}+1$ for every
$\eps>0$. Bansal~\etal~\cite{BHI+07} gave a 1.69-approximation
algorithm for the 3-dimensional version.

Further related problems are the {\em 2-dimensional knapsack} and {\em
  bin packing} problems. Given a set of axis-aligned rectangles and a
box $B$, the {\em geometric 2D knapsack} problem asks for a subset of
the rectangles of maximum total area that fit into $B$. In contrast,
the {\em 2D bin packing} asks for the minimum number of bins congruent to $B$
that can accommodate all rectangles. Jansen and Pr\"adel~\cite{JP}
designed a PTAS for the geometric 2D knapsack problem, although it does
not admit a FPTAS. The weighted version does not admit an AFPTAS and
an approximation algorithm by Jansen and Zhang~\cite{JZ07} guarantees
a ratio of $2+\eps$ for every $\eps>0$.
For the 2D bin packing, Jansen~\etal~\cite{JPS09}
gave a 2-approximation, and Bansal~\etal~\cite{BCS09} designed a
randomized algorithm with an {\em asymptotic} approximation ratio of
about $1.525$, improving the previous ratio $1.691$ by
Caprara~\cite{Cap02,Cap08}. However, 2D bin packing does not admit an
AFPTAS~\cite{BCKS06,CC09}.  Finally, we mention that
Bansal~\etal~\cite{BCKS06} gave a PTAS for the {\em rectangle
  placement} problem, which goes back to Erd\H{o}s and
Graham~\cite{EG75}. Here a given set of axis-aligned rectangles should be
arranged (without rotation or overlaps) so that the area of their
bounding box is minimized.

\section{Constructing a rectangle packing}

In this section we describe the two strategies for Bob
and then compare their performance.

\paragraph{Ordering the points in $S$.}
Let $S$ be a set of $n$ distinct points in the unit square $[0,1]^2$
such that $(0,0)\in S$. Denote by $x(s)$ and $y(s)$, respectively, the
$x$- and $y$-coordinates of each point $s\in S$. Order the points in
$S$ as $s_1,s_2,\ldots ,s_n$ such that
$$ x(s_j)+y(s_j)\leq x(s_i)+y(s_i) $$
for $1\leq i<j\leq n$ (ties are broken arbitrarily). 
Equivalently, this order is given by a left-moving sweep-line with slope $-1$.
See Fig.~\ref{fig:2}. Clearly, we have $s_n=(0,0)$. In {\sc GreedyPacking},
Bob chooses rectangles of maximum area for $s_1,\ldots , s_n$ in this order.

\begin{quote}
{\sc GreedyPacking.} For $i=1,\ldots , n$, choose an axis-aligned
rectangle $r_i\subseteq [0,1]^2$ of maximum area such that the lower
left corner of $r_i$ is $s_i$, and $r_i$ is interior-disjoint from any
$r_j$, $j<i$.
\end{quote}

Recall the partial order, called {\em dominance order}, among points
in the plane. For two points, $p=(x_p,y_p)$ and $q=(x_q,y_q)$, we say
that $p \preceq q$ (in words, {\em $q$ dominates $p$}) if
$$ x_p\leq x_q \hspace{1cm} \mbox{\rm and} \hspace{1cm} y_p\leq y_q. $$
With this definition, an axis-aligned rectangle with lower left
corner $c_1$ and upper right corner $c_2$
can be written as $\{p\in \RR^2: c_1\preceq p\preceq c_2\}$.
In particular, any point in $r(s)$ dominates $s$.

\begin{figure}[htbp]
\centering
\includegraphics[width=.9\textwidth]{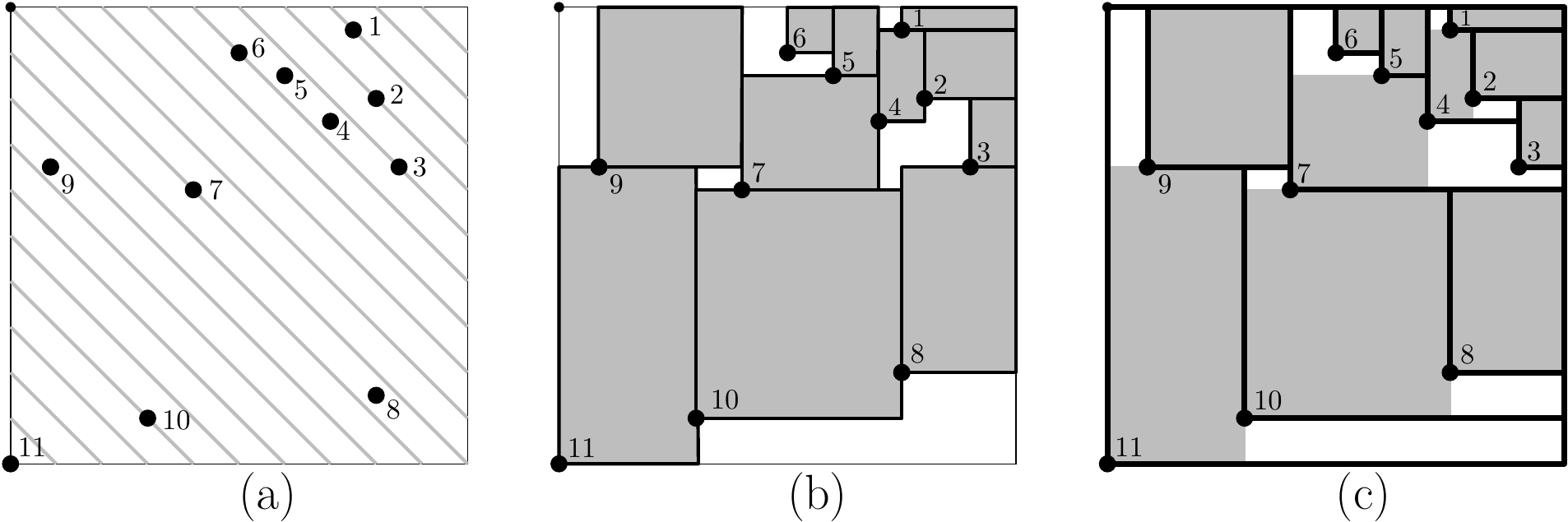}
\caption{(a) Eleven points in $[0,1]^2$, sorted by decreasing order of
  the sum of coordinates. 
(b) Rectangles chosen greedily in this order.
(c) The tiling of $[0,1]^2$ induced by the dominance order; and the
  maximum-area rectangles chosen from each tile. Note that rectangle 8
  is smaller than the corresponding greedy rectangle.} 
\label{fig:2}
\end{figure}

We now define interior-disjoint {\em tiles} for the set
$S=\{s_1,\ldots , s_n\}$ that jointly cover the unit square
$U=[0,1]^2$. For $i=1,2,\ldots, n$, let tile $t_i$ be the set of
points in $U$ that dominate $s_i$, but have not been covered by any
previous tile $t_j$, $j<i$. Formally, let
$$t_i=\{ p\in [0,1]^2 : s_i\preceq p \mbox{ \rm and } s_j\not\preceq p
\mbox{ \rm for all } j<i\}.$$

The tiles are disjoint by definition, and they cover $U$ since the
origin is in $S$. Each tile is a {\em staircase polygon} with
axis-aligned sides, bounded by one horizontal side from below,
one vertical side from the left, and a monotone decreasing
curve from the top and from the right. Observe that the axis-aligned
rectangle spanned by the lower left corner $s_i$ and any point $p\in
t_i$ is contained in the tile $t_i$. That means that every maximum-area
axis-aligned rectangle in the tile is incident to the lower left
corner $s_i$. We can now describe our second strategy for Bob.
\begin{quote}
{\sc TilePacking.} Compute the tiling $U=\bigcup_{i=1}^n t_i$. For
$i=1,\ldots , n$, independently, choose an axis-aligned rectangle
$r_i\subseteq t_i$ of maximum area.
\end{quote}
By the above observation, the lower left corner of $r_i$ is the
lower left corner of the tile $t_i$, which is $s_i\in S$.

We now show that {\sc GreedyPacking} always covers a greater
or equal area than {\sc TilePacking}.

\begin{lemma}\label{lem:greedy}
For each point $s_i\in S$, {\sc GreedyPacking} chooses a rectangle of
greater or equal area than {\sc TilePacking}.
\end{lemma}
\begin{proof}
The rectangles chosen by the greedy tiling for $s_j$, $j<i$, are all
disjoint from the tile $t_i$, because every point in
$r(s_j)$ dominates $s_j$. Hence, {\sc GreedyPacking} could choose
any maximum-area axis-aligned rectangle from tile $t_i$, but it may
choose a larger rectangle (such as $r(s_8)$ in Fig.~\ref{fig:2}).
\end{proof}

\paragraph{Remark.} It is worth noting that {\sc GreedyPacking} cannot
give a better worst-case constant than {\sc TilePacking}.
If $S$ is in "general position" (that is, no two points lie on the same 
sweep-line), we construct a set $S'$, where $|S'| \leq 3n-2$ as follows.
Choose a sufficiently small $\eps>0$. 
Then for each point $s_i=(x_i,y_i)$ in the interior of $U$, we add two
nearby points $(x_i-\eps, y_i)$ and $(x_i,y_i-\eps)$ 
(below the sweep-line incident to $s_i$, 
one to the left of $s_i$ and one below $s_i$). 
Observe that on input $S'$, {\sc GreedyPacking} and {\sc TilePacking}
give the same set of rectangles. Moreover, the total area covered by 
{\sc TilePacking} with $S'$ and the total area covered by {\sc TilePacking} 
with $S$ will differ from each other by an arbitrarily small amount,
\ie, by at most $2n\eps$.

\section{Analysis of {\sc TilePacking}}\label{sec:analysis}

Formally, define $\rho$ as the worst-case performance of {\sc TilePacking}.
In this section, we show that {\sc TilePacking} chooses a set of
rectangles of total area $\rho \geq 0.09121$. In our analysis,
we will use two variables, $\beta\geq 5$ and $0<\lambda<1$.

Let $r_i\subseteq t_i$ be the axis-aligned rectangle of maximum area
in tile $t_i$, whose lower left corner is $s_i$, and let
$R=\{r_i:i=1,\ldots , n\}$ be this set of rectangles. If
$\area(r_i)\geq 0.09121 \cdot \area(t_i)$ for every $i$, then our proof is
complete. However, the ratio $\area(r_i)/\area(t_i)$ may be arbitrarily small
because the tiles can be arbitrary staircase polygons.

For $\beta \geq 5$, we say that tile $t_i$ is a {\em $\beta$-tile}
if $\area(r_i)<\frac{1}{\beta}\ \area(t_i)$. We will give an upper bound
$F(\beta,\lambda)$ on the total area of $\beta$-tiles for every $\beta \geq 5$
and $0<\lambda<1$. This immediately implies that, for every $\beta
\geq 5$, the complement of all $\beta$-tiles cover at least
$1-F(\beta,\lambda)$ area, and the total area of all rectangles in $R$
is
\begin{equation}\label{eq:total}
\sum_{i=1}^n \area(r_i)
\geq  \rho
\geq \frac{1-F(\beta,\lambda)}{\beta}.
\end{equation}

In Section~\ref{ssec:onetile} we study the properties of
individual $\beta$-tiles, and in Section~\ref{ssec:manytiles} we prove an
upper bound on the total area of $\beta$-tiles (for every $\beta \geq 5$),
which already gives a preliminary bound $\rho \geq 0.07229$, using \eqref{eq:total}.
By integrating over $\beta$, we improve this bound to $\rho \geq 0.09121$
in Section~\ref{ssec:int}. Finally, Section~\ref{ssec:runtikme} shows that {\sc TilePacking}
runs in $O(n\log n)$ time.

\subsection{Properties of a $\beta$-tile\label{ssec:onetile}}

Let us introduce some notation for describing a single tile $t_i$
(Fig.~\ref{fig:3}(a)). It is bounded from below by a horizontal side,
denoted $a_i$, and from the left by a vertical side, denoted
$b_i$. The {\em width} (resp., {\em height}) of $t_i$ is the length of
$a_i$ (resp., $b_i$), denoted by $|a_i|$ (resp., $|b_i|$). We show
next that a $\beta$-tile $t_i$ has a much smaller area than its
bounding box (recall that $\beta$-tiles are defined for $\beta \geq 5$).

\begin{lemma}\label{lem:sides}
Let $\beta \geq 1$ and $t$ be a staircase polygon of height $h$ and width $w$.
If the area of every axis-aligned rectangle contained in $t$ is less
than $\area(t)/\beta$, then
\begin{equation}\label{eq:sides}
\area(t)< \frac{\beta}{e^{\beta-1}} \cdot hw,
\end{equation}
and this bound is the best possible.
\end{lemma}
\begin{proof}
Assume, by translating $t$ if necessary, that the lower left corner of
$t$ is the origin. Then the bounding box of $t$ is
$[0,w]\times [0,h]$. Let $\area(t_i)=\beta u^2$, for some $u>0$. Then
all vertices of $t$ lie strictly below the hyperbola arc $f(x)=u^2/x$,
for $0<x$. See Fig.~\ref{fig:3}(b).
The area of the part of $[0,w]\times [0,h]$ below the curve $f(x)=u^2/x$ is
$$u^2+\int_{u^2/h}^{w}\frac{u^2}{x}\ {\rm d}x
= u^2+\big[u^2\ln x\big]_{u^2/h}^{w}
= u^2(1+\ln w- \ln (u^2/h))
= u^2(1+\ln(hw)-\ln u^2).$$
We have shown that $\area(t)=\beta u^2 < u^2(1+\ln (hw)-\ln u^2)$, or
$\beta<1+ \ln (hw)-\ln u^2$.
Rearranging this inequality yields $e^{\beta-1}u^2<hw$, which implies
\eqref{eq:sides}, as required.

If we approximate the shaded area in Fig.~\ref{fig:3}(b) with a
staircase polygon $t$ of height $h$ and width $w$ that lies strictly
below the hyperbola, then the area of every axis-aligned rectangle
contained in $t$ is less than $u^2$, and $\area(t)$ can be arbitrarily
close to $\frac{\beta}{e^{\beta-1}}\cdot hw$.
\end{proof}

\paragraph{Sectors and tips.}
Fix $\beta \geq 5$ and consider a $\beta$-tile $t_i$. Decompose $t_i$
into rectangular {\em vertical sectors} by vertical lines passing though
its vertices; see Fig.~\ref{fig:3}(c). Each sector is part of some maximum-area
axis-aligned rectangle in $t_i$. Hence the area of each sector is less than
$\area(t_i)/\beta$, where near equality is possible for the leftmost
sector. Similarly, we can decompose $t_i$ into rectangular
{\em horizontal sectors} by horizontal lines passing through its
vertices, and the area of each sector is less than $\area(t_i)/\beta$.
\begin{figure}[htbp]
\centering
\includegraphics[width=.95\textwidth]{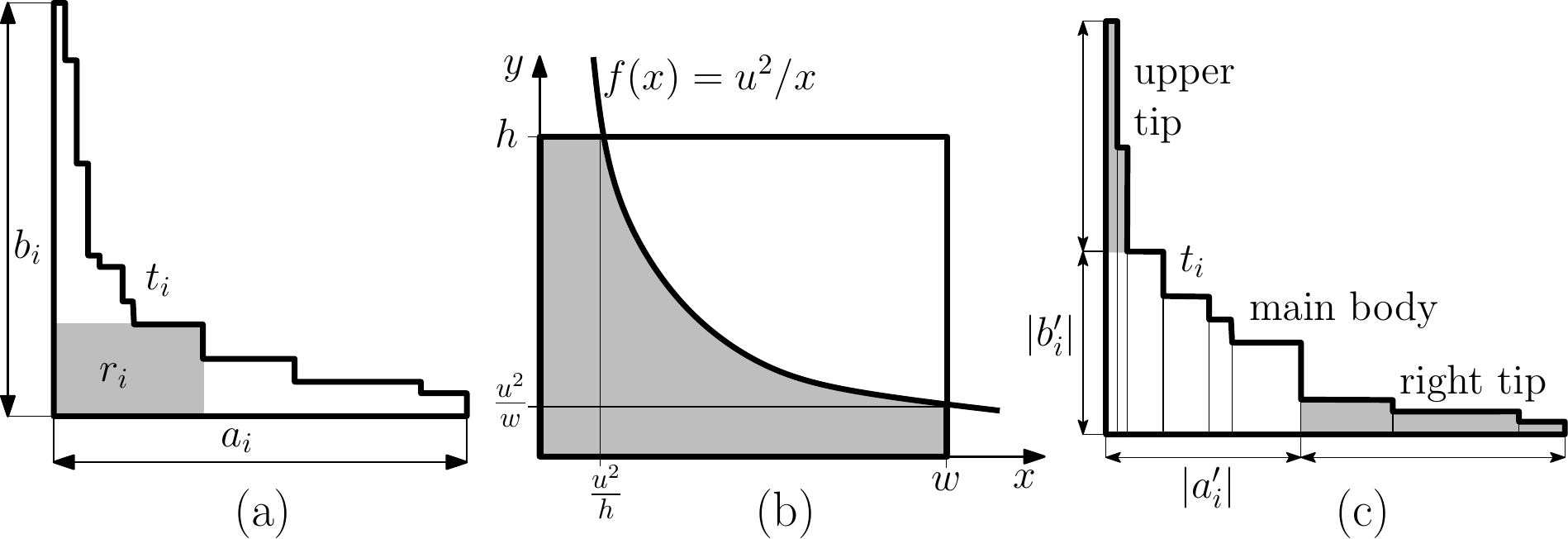}
\caption{(a) A $\beta$-tile $t_i$ of width $|a_i|$ and height $|b_i|$.
(b) The portion of the rectangle $[0,w] \times [0,h]$ below the
  hyperbola arc $f(x)=u^2/x$.
(c) The decomposition of $t_i$ into vertical sectors. The tips of
  $t_i$ are shaded.}
\label{fig:3}
\end{figure}

Decompose each $\beta$-tile $t_i$ into three parts, called {\em right tip},
{\em upper tip}, and {\em main body}, as follows. The right
tip of $t_i$ is cut off from $t_i$ by the right-most vertical line
that passes through a vertex of $t_i$ such that the area of the right
part is at least $\area(t_i)/\beta$. 
Similarly, the upper tip of $t_i$ is cut off from
$t_i$ by the upper-most horizontal line that passes through a vertex
of $t_i$ such that the area of the part above is at least
$\area(t_i)/\beta$. The remaining part, denoted by $t_i'$,
is the main body of $t_i$. All three parts are staircase
polygons. Both the right and the upper tips are unions of some sectors
of $t_i$. Since the area of each sector is less than $\area(t_i)/\beta$,
the area of each tip is at least $\frac{1}{\beta}\area(t_i)$ but less than
$\frac{2}{\beta}\area(t_i)$. In particular, since $\beta\geq 5$, the right
tip of $t_i$ is disjoint from the upper tip of $t_i$.

Let $a_i'$ and $b_i'$, respectively, be the lower and left side of
$t_i'$. Note that the topmost horizontal side and the rightmost
vertical side of $t_i'$ each contains some point from $S$,
because each contains a reflex vertex of the original tile $t_i$.

\begin{lemma}\label{lem:tip}
The width of the right tip of $t_i$ is at least $|a_i'|$, hence $2|a_i'|\leq |a_i|$.
Similarly, the height of the upper tip of $t_i$ is at least $|b_i'|$,
hence $2|b_i'|\leq |b_i|$.
\end{lemma}
\begin{proof}
By symmetry, it is enough to prove the first claim. Let $r'$ be the
maximum-area axis-aligned rectangle in $t_i$ whose lower side is $a_i'$.
Since $t_i$ is a $\beta$-tile, the area of $r'$ is less than
$\frac{1}{\beta}\area(t_i)$.  Recall that the area of each tip is
at least $\frac{1}{\beta}\area(t_i)$, thus $\area(r')$ is less than
the area of the right tip of $t_i$, and so $\area(r')$ is less than
the area of the bounding box of the right tip of $t_i$. However, the
height of $r'$ is strictly greater than the height of the right tip
(and its bounding box). Therefore, the width of $r'$, which is
$|a_i'|$, is less than the width of the right tip of $t_i$.
Thus $|a_i'| \leq |a_i|- |a_i'|$, or $2|a_i'| \leq |a_i|$, as required.
\end{proof}

Recall that the areas of the right and upper tips of $t_i$ are each less than
$\frac{2}{\beta}\area(t_i)$. Hence $\area(t_i) <\frac{\beta}{\beta-4} \area(t_i')$.
Since $t_i$ is a $\beta$-tile and $t_i' \subset t_i$,
the area of every axis-aligned rectangle contained
in $t_i'$ is less than $\area(t_i)/\beta< \area(t_i')/(\beta-4)$.
Applying Lemma~\ref{lem:sides} for the main body $t_i'$ yields
\begin{equation}\label{eq:mainbody}
\area(t_i)
<\frac{\beta}{\beta-4} \cdot \area(t_i')
<\frac{\beta}{\beta-4} \cdot \frac{\beta-4}{e^{\beta-5}} \cdot |a_i'|\cdot |b_i'|
= \frac{\beta}{e^{\beta-5}} \cdot |a_i'|\cdot |b_i'|.
\end{equation}

\paragraph{Tall and wide tiles.}
We distinguish two types of $\beta$-tiles based on the height and width of
their main body. A $\beta$-tile $t_i$ is {\em tall} if $|a_i'|<|b_i'|$;
and it is {\em wide} if $|a_i'|\geq |b_i'|$. For wide $\beta$-tiles,
we have $\max(|a_i'|,|b_i'|)=|a_i'|$, and \eqref{eq:mainbody} implies
\begin{equation}\label{eq:b5}
\area(t_i)< \frac{\beta}{e^{\beta-5}} |a_i'|^2.
\end{equation}
Similarly, if a $\beta$-tile $t_i$ is tall, then 
$\area(t_i)< \frac{\beta}{e^{\beta-5}} |b_i'|^2$.

\subsection{Upper bound on the total area of $\beta$-tiles\label{ssec:manytiles}}

In this section we give an upper bound $F(\beta,\lambda)$
(see equation~\eqref{eq:F} further bellow) on the total area of all $\beta$-tiles
for every $\beta\geq 5$ and $0<\lambda<1$. It is enough to  bound the
total area of wide $\beta$-tiles by $\frac{1}{2}F(\beta,\lambda)$. By symmetry, the same
upper bound holds for the total area of tall $\beta$-tiles. Let $W\subset \{1,\ldots n\}$
be the set of indices of the wide $\beta$-tiles.

To begin, for every tile $t_i$ we define two adjacent triangles. Let $\Delta_i$ be the
isosceles right triangle bounded by $a_i$, the line of slope $-1$ through $s_i$, and
a vertical line through the right endpoint of $a_i$ (see Fig.~\ref{fig:4}(a)).
Similarly, let $\Gamma_i$ be isosceles right triangle adjacent to
$b_i$ that lies left of $t_i$. The two triangles $\Delta_i$ and $\Gamma_i$
are not part of the tiling $\{t_i:i=1\ldots n\}$, and they may intersect several tiles.
A key fact is that these two triangles are empty of points from $S$ in
their interior, regardless whether the tile $t_i$ is a $\beta$-tile or not.

\begin{lemma}\label{lem:empty}
For every $i=1,\ldots , n$, the interior of $\Delta_i$
(resp., $\Gamma_i$) is disjoint from $S$.
\end{lemma}
\begin{proof}
Suppose to the contrary, that there exist points in $S$ in the interior of
$\Delta_i$, and let $s_j$ be the first such point processed by the algorithm. 
Then $s_j$ is processed before $s_i$, and so the left side $b_j$ of the tile $t_j$
would cut through the horizontal segment $a_i$, which is a
contradiction. Similarly, if $s_j$ lies in the interior of $\Gamma_i$,
then the lower side $a_j$ would cut through the left side $b_i$.
\end{proof}

\begin{figure}[htbp]
\centering
\includegraphics[width=.95\textwidth]{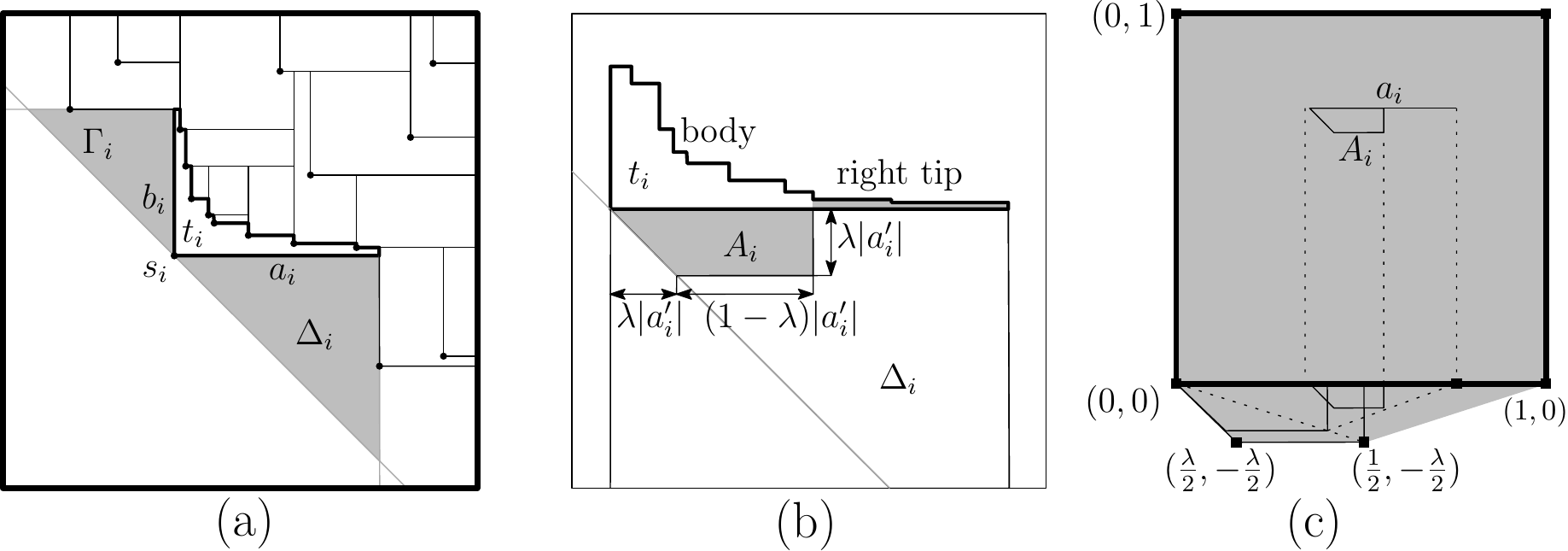}
\caption{(a) Tiles $t_1,\ldots, t_i$, and the triangles $\Delta_i$ and $\Gamma_i$.
(b) The trapezoid $A_i\subset \Delta_i$.
(c) Every trapezoid $A_i$ is contained in the shaded (hexagonal) region.}
\label{fig:4}
\end{figure}

We charge the area of each wide $\beta$-tile $t_i$ to the trapezoid $A_i\subset
\Delta_i$ defined below. Let $A_i$ be the set of points in $\Delta_i$
that lie vertically below the segment $a_i'$ at distance at most $\lambda |a_i'|$ from it.
See Fig.~\ref{fig:4}(c). Using Inequality~\eqref{eq:b5}, the area of $A_i$
can be bounded from below as follows in terms of the area of $t_i$:
\begin{equation}\label{eq:trapezoid}
\area(A_i)
=\frac{|a_i'|+(1-\lambda)|a_i'|}{2} \cdot \lambda |a_i'|
= \frac{\lambda(2-\lambda)}{2} \cdot |a_i'|^2
> \frac{\lambda(2-\lambda) \cdot e^{\beta-5}}{2\beta} \cdot \area(t_i).
\end{equation}

The trapezoids $A_i$, $i\in W$, are homothetic copies of each other,
and their area depends only on $|a_i'|$. Note that the triangles
$\Delta_i$ (and also the trapezoids $A_i$) may extend beyond the boundary of
$U=[0,1]^2$ ({\em e.g.}, in Fig.~\ref{fig:4}(b), $\Delta_i$ extends below $U$).
We show that all trapezoids $A_i$, $i\in W$, lie in a polygon whose area
is at most $1+\lambda(3-\lambda)/8$.

\begin{lemma}\label{lem:AI}
Every trapezoid $A_i$, $i\in W$, lies in a polygon of area $1+\lambda(3-\lambda)/8$.
\end{lemma}
\begin{proof}
Every trapezoid $A_i$ lies vertically below a segment $a_i'$, which is
part of the lower side of tile $t_i$. Therefore, $A_i$ cannot extend beyond
the left, right, and upper sides of the unit square $U$. Moreover, we show that $A_i$
is contained in the shaded polygon in Fig.~\ref{fig:4}(c).

Consider a trapezoid $A_i$, and the corresponding lower side $a_i$ of a
$\beta$-tile $t_i$, $i\in W$. Translate them vertically down until $a_i$
lies on the $x$-axis. Then apply a dilation centered at the right endpoint
of $a_i$ (now on the $x$-axis), such that the left endpoint of $a_i$
becomes $(0,0)$. Finally, apply a dilation centered at $(0,0)$ such that
the right endpoint of $a_i$ becomes $(1,0)$. Observe that through all
three transformations, the trapezoid $A_i$ remains in the shaded polygon.
The shaded polygon is the union of $U$ and a trapezoid of area
$\frac{1}{2}(1+\frac{1-\lambda}{2}) \frac{\lambda}{2}
=\lambda(3-\lambda)/8$.
\end{proof}

\paragraph{The case of pairwise disjoint trapezoids.}
If the trapezoids $A_i$, for all $i\in W$, are pairwise disjoint,
then we can deduce an upper bound for the total area of wide $\beta$-tiles:
By Lemma~\ref{lem:AI}, we have
$$\sum_{i\in W} \area(A_i) \leq 1+\frac{\lambda(3-\lambda)}{8} =
\frac{8+3\lambda-\lambda^2}{8}.$$
Using Inequalities \eqref{eq:b5} and \eqref{eq:trapezoid}, this implies
\begin{equation}\label{eq:no-overlap}
\sum_{i\in W}\area(t_i)
< \frac{8+3\lambda-\lambda^2}{8} \cdot
\frac{2\beta}{\lambda(2-\lambda) \cdot e^{\beta-5}}
=\frac{(8+3\lambda-\lambda^2)}{4\lambda(2-\lambda)}  \cdot \frac{\beta}{e^{\beta-5}}.
\end{equation}

\paragraph{The general case of overlapping trapezoids.}
However, it is possible that the trapezoids $A_i$, $i\in W$, are not
disjoint. To take care of this possibility, we set up a charging
scheme, in which we choose a set of ``large'' pairwise disjoint
trapezoids. For every trapezoid $A_i$, $i\in W$, denote by $\ell_i$ the 
supporting line of $a_i$. We say that $A_i$ is {\em above} $A_j$ (and $A_j$ is
{\em below} $A_i$) if $A_i$ and $A_j$, $i\neq j$ and $\ell_i$ is 
above $\ell_j$. We next show that if $A_i$ and $A_j$ overlap, 
then the trapezoid below the other is significantly larger.

\begin{lemma}\label{lem:exp}
Assume that $A_i\cap A_j\neq \emptyset$, for some $i,j\in W$, 
$i\neq j$; and $A_i$ is above $A_j$. Then $\ell_j \cap \Delta_i \subseteq a_j'$ 
and $(2-\lambda)|a_i'|\leq |a_j'|$.
\end{lemma}
\begin{proof}
Refer to Fig.~\ref{fig:5}(a). If $A_i\cap A_j\neq \emptyset$, and 
line $\ell_i$ is above line $\ell_j$, then the segment $a_j'$ has 
to intersect $A_i$. Note that the left endpoint of $a_j'$ is $s_j\in S$, and 
there is some point from $S$ on the rightmost edge of $t_j'$. 
By Lemma~\ref{lem:empty}, however, there is no point from $S$ in
the interior of $\Delta_i$. Therefore, $a_j'$ has to traverse both
$A_i$ and $\Delta_i$, hence $\ell_j \cap \Delta_i \subseteq a_j'$.

The minimum horizontal cross-section of $A_i$ is
$(1-\lambda)|a_i'|$, and the width of the right tip of $t_i$ is at
least $|a_i'|$ by Lemma~\ref{lem:tip}. It follows that 
$|a_j'| \geq |\ell_j \cap \Delta_i| \geq (1-\lambda)|a_i'|+|a_i'| = (2-\lambda)|a_i'|$.
\end{proof}

\begin{figure}[htbp]
\centering
\includegraphics[width=.9\textwidth]{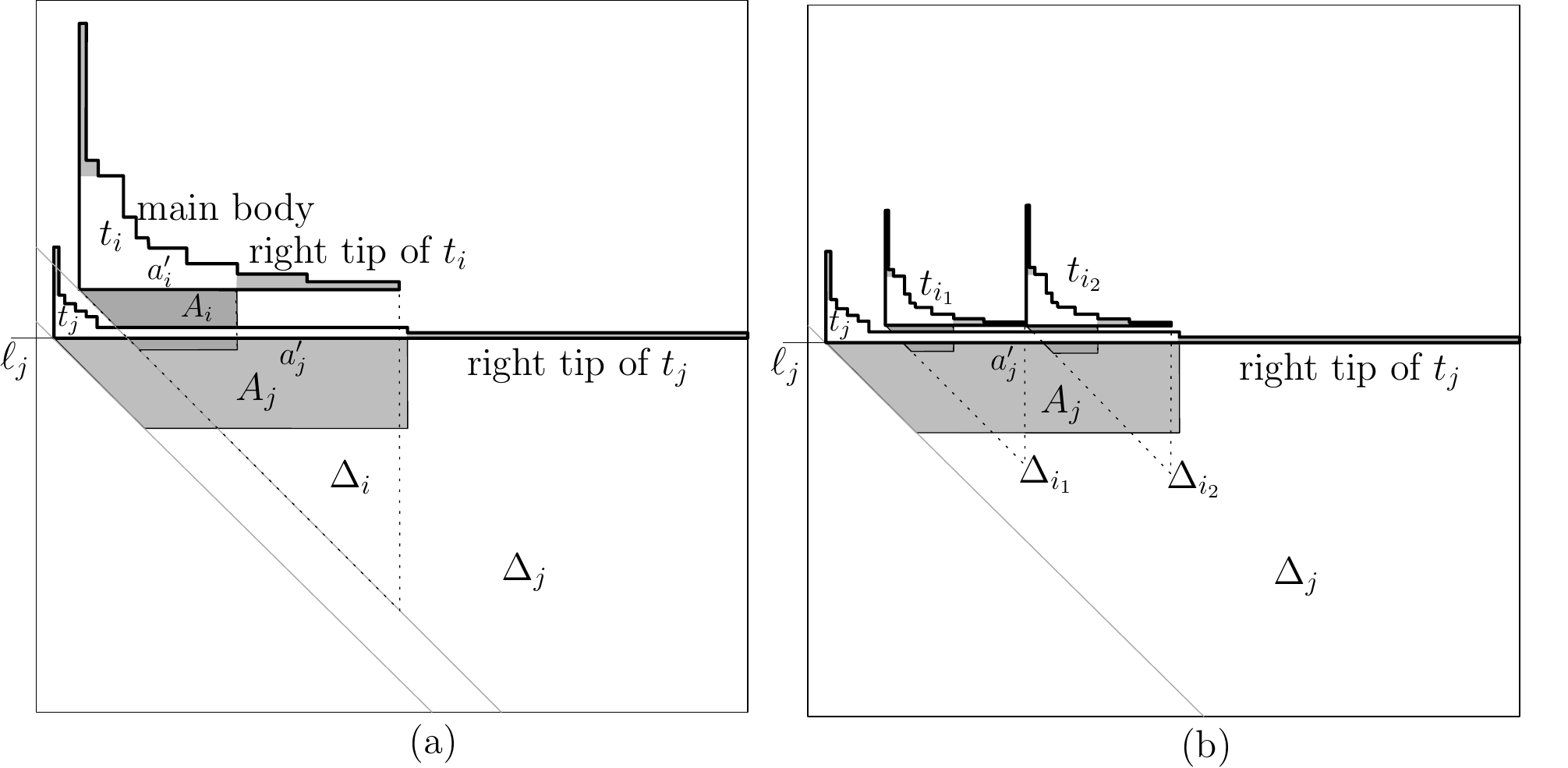}
\caption{(a) If $A_i$ and $A_j$ overlap, and $a_j$ lies below $a_i$, then
  $|a_j'|\geq (2-\lambda)|a_i'|$.
(b) Trapezoids $A_{i_1}$ and $A_{i_2}$ intersect $A_j$ from above
such that the intervals  $a_j'\cap \Delta_{i_1}$ and
$a_j'\cap \Delta_{i_2}$ are disjoint.}
\label{fig:5}
\end{figure}

\paragraph{Charging scheme.}
We introduce a charging scheme among the trapezoids $A_i$, $i\in W$.
Initially, each trapezoid $A_i$ has a charge of $\area(A_i)$.
We transfer the charges to a subset of pairwise disjoint
trapezoids. The transfer of charges is represented by a directed
acyclic graph $G$ defined as follows.
The nodes of $G$ correspond to the trapezoids $A_i$, $i\in W$.
If $A_i$ intersects some other trapezoid below, we add a unique outgoing edge
from $A_i$ to the trapezoid $A_j$, $j\in W$, whose top side $a_j'$ is the highest
below $a_i'$. Observe that all edges of $G$ are oriented downwards,
thus $G$ is acyclic. By construction, the out-degree of $G$ is at most one.
However, the in-degree of a node in $G$ may be higher than one.

\begin{lemma} \label{lem:charge}
For every trapezoid $A_j$, the total area of all trapezoids $A_i$, $i\neq j$, 
with a directed path in $G$ to $A_j$ is at most
$\frac{1}{(1-\lambda)(2-\lambda)} \area(A_j)$.
\end{lemma}
\begin{proof}
Fix a trapezoid $A_j$, $j\in W$, and denote by $\ell_j$ be the 
supporting line of $a_j$. Refer to Fig.~\ref{fig:7}(a). 
For $k\geq 1$, let $W_j^k\subseteq W$ be the set of indices $i$ of 
trapezoids $A_i$ that have a directed path of length exactly $k$ to $A_j$ in $G$. 
We say that the trapezoids $A_i$, $i\in W_j^k$ are {\em on level $k$}.
In particular, the trapezoids $A_i$, $i\in W_j^1$, on level 1 are connected to 
$A_j$ with a directed edge $(A_i,A_j) \in G$. 
Let $W_j = \bigcup_{k\geq 1}W_j^k$ be the set of indices of {\em all} trapezoids 
$A_i$, $i\neq j$, with a directed path in $G$ to $A_j$. For each $A_i$, $i\in W_j$,
denote by $A_i^*$ the unique trapezoid with $(A_i,A_i^*)\in G$.

By Lemma~\ref{lem:exp}, every trapezoid $A_i$, $i\in W_j^1$, has 
width at most $|a_i'|\leq |a_j'|/(2-\lambda)$, and height at most 
$\frac{\lambda}{2-\lambda}|a_j'|$. Equality is possible if the 
lower left corner of $A_i$ coincides with the upper left corner $s_j$ of $A_j$
(see $A_{i_1}$ in Fig.~\ref{fig:7}(a)). The gray triangle in Fig.~\ref{fig:7}(a) 
is the minimum triangle with base $a_j'$ that contains the maximal possible 
trapezoid intersecting $A_j$ from above. By triangle similarity, 
the height of this triangle is $\frac{\lambda}{1-\lambda}|a_j'|$, 
hence its area is  (using Equation~\eqref{eq:trapezoid})
\begin{equation}\label{eq:triangle}
\frac{1}{2}\cdot |a_j'|\cdot \frac{\lambda}{1-\lambda}|a_j'|
=\frac{\lambda}{2(1-\lambda)}|a_j'|^2
=\frac{\area(A_j)}{(1-\lambda)(2-\lambda)}.
\end{equation}
\begin{figure}[htbp]
\centering
\includegraphics[width=.95\textwidth]{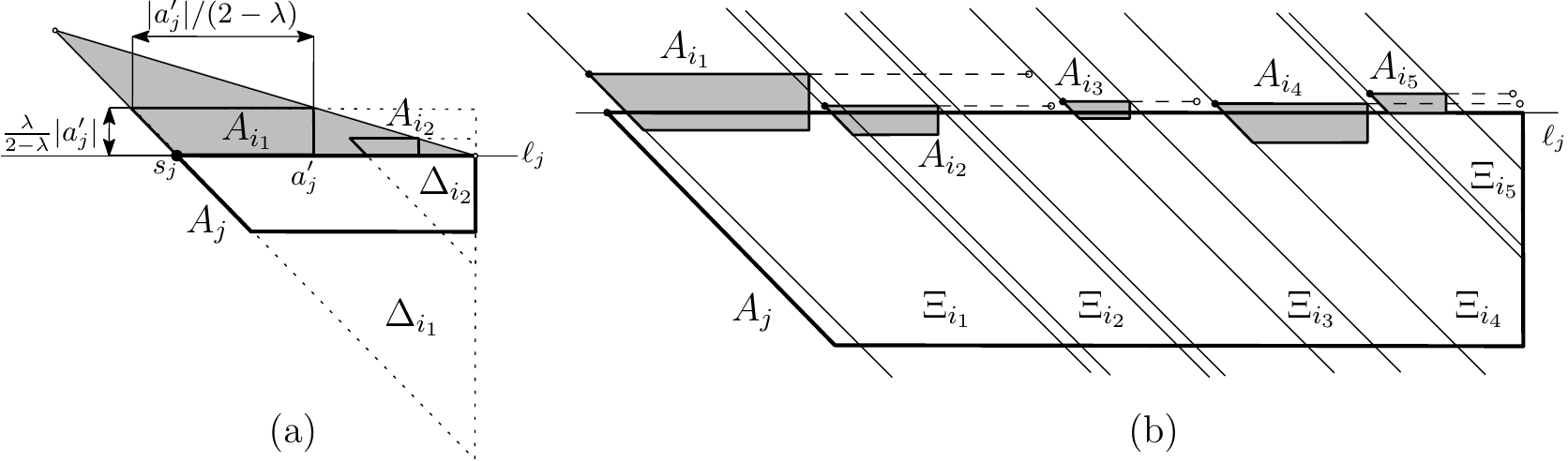}
\caption{(a) $A_{i_1}$ is the largest possible trapezoid
that can intersect $A_j$ from above. All trapezoids with a 
directed path in $G$ to $A_j$ can be translated (without overlaps) 
into the gray triangle.
(b) Disjoint trapezoids $A_{i_1}\ldots , A_{i_5}$ intersect $A_j$ from
above. Each of these trapezoids induces a parallel strip $\Xi_i$. If
the strips $\Xi_{i_1}$ and $\Xi_{i_2}$ intersect, and $A_{i_1}$ is above
$A_{i_2}$, then $s_{i_2}$ lies in the interior of $\Delta_{i_1}$.
Dashed lines indicate segments $a_i\setminus a_i'$.}
\label{fig:7}
\end{figure}

We claim that $\sum_{i\in W_j}\area(A_i)$ is at most the area of
the gray triangle in Figure~\ref{fig:7}(a).
To verify the claim, we translate every trapezoid $A_i$, $i\in W_j$,
into the gray triangle region such that they remain pairwise disjoint. Each trapezoid
will be translated in the same direction, $(-1,1)$, but at different distances. In
order to control the possible location of the translates, we enclose each $A_i$, 
$i\in W_j$, in a parallel strip. 
For every $i\in W_j$, draw lines of slope $-1$ through the two endpoints of $a_i'$,
and denote by $\Xi_i$ the strip bounded by the two lines. Refer to Fig.~\ref{fig:7}(b).

First consider the trapezoids $A_i$, $i\in W_j^1$, which are connected
to $A_j$ by a directed edge in $G$. By the definition of $G$, the
trapezoids $A_i$, $i\in W_j^1$, are pairwise disjoint. Label their
strips in increasing order from left to right. We show that the strips 
$\Xi_i$, $i\in W_j^1$, are pairwise interior-disjoint. Suppose to the
contrary that $\Xi_{i_1}$ and  $\Xi_{i_2}$ intersect, where
${i_1}<{i_2}$, and such that $A_{i_1}$ is above $A_{i_2}$ (if
$A_{i_1}$ is below $A_{i_2}$, the strips are obviously disjoint). 
Since the trapezoids are disjoint, the left endpoint of
$a_{i_2}'$, $s_{i_2}\in S$, lies in the interior of the
isosceles right triangle bounded by the right side of
$A_{i_1}$, line $\ell_j$, and the line of slope $-1$ bounding the strip
$\Xi_{i_1}$ from the right. Since $\lambda<1$
and $|a_{i_1}|-|a_{i_1}'| \geq |a_{i_1}'|$ by Lemma~\ref{lem:tip},
this triangle is contained in $\Delta_{i_1}$. However, the triangle
$\Delta_{i_1}$ is empty of points from $S$, and we reached a contradiction.

Applying the above argument for the trapezoids on level $k$, $k=1,2,\ldots$, 
we conclude that the trapezoids on level $k+1$ are pairwise disjoint, and 
they induce pairwise disjoint parallel strips. 

We are now ready to describe the translation of the trapezoids $A_i$, $i\in W_j$. 
We translate the trapezoids in direction $(-1,1)$ in phases $k=1,2,\ldots$. In phase $k$,
consider each $A_i$, $i\in W_j^k$, independently. Translate $A_i$ together with all other
trapezoids that have a directed path to $A_i$ by the same vector in direction $(-1,1)$
until the lower side of $A_i$ becomes collinear with the upper side of $A_i^*$.
Each trapezoid $A_i$ remains in its parallel strip $\Xi_i$, therefore the trapezoids
on the same level remain pairwise disjoint. After phase $k$, there is no overlap 
between a trapezoid $A_i$, $i \in W_j^k$ of level $k$ and trapezoids in lower levels. 
When all phases are complete, all trapezoids are pairwise interior-disjoint. 

It remains to show that after the translation, all trapezoids lie in the 
gray triangle in Fig.~\ref{fig:7}(a). From Lemma~\ref{lem:exp} and since 
all translations were done in direction $(-1,1)$, the  trapezoids are 
on or to the right of the line of slope $-1$ passing through $s_j$. 
Also from Lemma~\ref{lem:exp}, the right endpoint of $a_j'$ is to the 
right of the right side of every triangle $\Delta_i$, $i\in W_j^1$, hence 
to the right of the right side of every triangle $\Delta_i$, $i\in W_j$. 
Also, if $i\in W_j^1$, by Lemma~\ref{lem:tip} we have
$2|a_i'| \leq |a_i|$, and so
$\frac{\lambda |a_i'|}{|a_i|-|a_i'|} \leq \frac{\lambda |a_i'|}{|a_i'|} =\lambda$.
Hence the upper right corner of every $A_i$, $i\in W_j^1$, is below
the line of slope $-\lambda$ passing through the right endpoint of $a_j'$, 
and consequently, the upper right corner
of every $A_i$, $i\in W_j$, is below the line of slope $-\lambda$
passing through the right endpoint of $a_j'$.
Therefore, after the above translation, all trapezoids $A_i$, $i\in W_j$,
are contained in the gray triangle in Figure~\ref{fig:5}(c). 
This verifies the above claim and completes the proof of the lemma.
\end{proof}

Transfer the charges from all nodes to the sinks in $G$ along directed paths.
By Lemma~\ref{lem:charge}, the total area charged to a sink $A_j$ is less than
\begin{equation}\label{eq:geometric}
\area(A_j)+\frac{1}{(1-\lambda)(2-\lambda)}\area(A_j)=
\frac{3-3\lambda+\lambda^2}{(1-\lambda)(2-\lambda)}\area(A_j)
\end{equation}

The area of every trapezoid $A_i$, $i\in W$, is charged to some sink in $G$,
and the sinks correspond to pairwise disjoint trapezoids. We can
now adjust Inequality~\eqref{eq:no-overlap} to obtain
\begin{equation}\label{eq:overlap}
\sum_{i\in W}\area(t_i)
< \frac{3-3\lambda+\lambda^2}{(1-\lambda)(2-\lambda)} \cdot
\frac{(8+3\lambda-\lambda^2)}{4\lambda(2-\lambda)}
\cdot \frac{\beta}{e^{\beta-5}}
= \frac{(3-3\lambda+\lambda^2)(8+3\lambda-\lambda^2)}{4\lambda(1-\lambda)(2-\lambda)^2}
\cdot \frac{\beta}{e^{\beta-5}}.
\end{equation}

The area of {\em all} $\beta$-tiles is less than twice the right hand-side
of \eqref{eq:overlap}, namely we can set
\begin{equation}\label{eq:F}
F(\beta,\lambda)
=\frac{(3-3\lambda+\lambda^2)(8+3\lambda-\lambda^2)}{2\lambda(1-\lambda)(2-\lambda)^2}
\cdot \frac{\beta}{e^{\beta-5}}.
\end{equation}

From \eqref{eq:total}, it follows that the total area of all
rectangles in $R$ is
\begin{align*}
\rho \geq \frac{1-F(\beta,\lambda)}{\beta}
&=\frac{1}{\beta}\left(1-
\frac{(3-3\lambda+\lambda^2)(8+3\lambda-\lambda^2)}{2\lambda(1-\lambda)(2-\lambda)^2}
\cdot \frac{\beta}{e^{\beta-5}}\right)\\
&=\frac{1}{\beta} -
\frac{(3-3\lambda+\lambda^2)(8+3\lambda-\lambda^2)}{2\lambda(1-\lambda)(2-\lambda)^2}
\cdot \frac{1}{e^{\beta-5}}.
\end{align*}

Whenever $F(\beta,\lambda)<1$, this already gives a lower bound of
$\rho=\Omega(1)$. We have optimized the parameters $\beta$ and
$\lambda$ with numerical methods. 
With the choice of $\beta=12.75$ and $\lambda=0.45$, we obtain an initial
lower bound of $\rho \geq 0.07229$.

\subsection{Making $\beta$  a continuous variable } \label{ssec:int}

In this section we further improve the lower bound on the covered area
to $0.09121$ by making $\beta$  a continuous variable and using integration.
We define the {\em contribution} of each point $p\in t_i\subseteq U$, as
$u(p)=\area(r_i)/\area(t_i)$. With this definition, we have
$$\sum_{i=1}^n\area(r_i)=\iint_{p\in U} u(p)\ {\rm d}A.$$

Let $\beta_0 \geq 5$ be a parameter to be optimized later (we will choose $\beta_0=9.955$).
Partition the interval $[\beta_0,\infty)$ into subintervals of length $\eps>0$:
$[\beta_0,\infty)=\bigcup_{j=1}^{\infty} [\beta_{j-1},\beta_j)$, where $\beta_j=\beta_0+j\eps$.
Denote by $B_j\subset U$ the union of all $\beta_j$-tiles, 
and let $\overline{B}_j=U\setminus B_j$.
By definition, we have $u(P)\geq 1/\beta_j$ for every $p\in \overline{B}_j$, and so
$\iint_{p\in \overline{B}_j} u(p)\ {\rm d}A \geq \area(\overline{B}_j)/\beta_j$.

Observe that the sets $\overline{B}_j$ form a nested sequence
$\overline{B}_0\subseteq \overline{B}_1\subseteq \overline{B}_2\subseteq \ldots\subseteq U$.
The total contribution of all points in $U$ can be written as
\begin{eqnarray}
\sum_{j=1}^n\area(r_i)
&=&\iint_{p\in U} u(p)\ {\rm d}A \nonumber\\
&=& \iint_{p\in \overline{B}_0} u(p)\ {\rm d}A
    + \sum_{j=1}^{\infty} \iint_{p\in \overline{B}_j\setminus \overline{B}_{j-1}} u(p)\ {\rm d}A\nonumber\\
&\geq& \frac{\area(\overline{B}_0)}{\beta_0}
    + \sum_{j=1}^{\infty} \frac{\area(\overline{B}_j\setminus \overline{B}_{j-1})}{\beta_j}\nonumber\\
&=& \frac{\area(\overline{B}_0)}{\beta_0}
    + \sum_{j=1}^{\infty} \frac{\area(\overline{B}_j)-\area(\overline{B}_{j-1})}{\beta_j}\nonumber\\
&=& \sum_{j=0}^{\infty}
     \area(\overline{B}_j)\left(\frac{1}{\beta_j}-\frac{1}{\beta_{j+1}}\right).\nonumber
\end{eqnarray}

In Section~\ref{ssec:manytiles}, we showed that for any $j \geq 0$,
$\area(B_j)<F(\beta_j,\lambda)$, where $F(\beta,\lambda)$ is given
by~\eqref{eq:F}. It follows that
$\area(\overline{B}_j)>1-F(\beta_j,\lambda)$, and therefore,
\begin{eqnarray}
\rho&\geq&\sum_{j=0}^\infty
(1-F(\beta_j,\lambda))\left(\frac{1}{\beta_j}-\frac{1}{\beta_{j+1}}\right)\nonumber\\
&=&\frac{1-F(\beta_0,\lambda)}{\beta_0}+
\sum_{j=0}^\infty \frac{(1-F(\beta_{j+1},\lambda)) -(1-F(\beta_j,\lambda))}{\beta_{j+1}}\nonumber\\
&=&\frac{1-F(\beta_0,\lambda)}{\beta_0}+
\sum_{j=0}^\infty \frac{F(\beta_j,\lambda) -F(\beta_{j+1},\lambda)}{\beta_{j+1}}\nonumber\\
&=&\frac{1-F(\beta_0,\lambda)}{\beta_0}-
\sum_{j=0}^\infty \frac{1}{\beta_{j+1}}\cdot
\frac{F(\beta_{j+1},\lambda)-F(\beta_j,\lambda)}{\beta_{j+1}-\beta_j}\cdot
(\beta_{j+1}-\beta_j). \nonumber
\end{eqnarray}

Letting $\eps$ go to 0 yields
\begin{eqnarray*} 
\rho&\geq&
\frac{1-F(\beta_0,\lambda)}{\beta_0}-
\int_{\beta_0}^{\infty} \left(\frac{1}{\beta} \cdot
\frac{\partial}{\partial \beta}F(\beta,\lambda)\right) {\rm d} \beta\\
&=&\frac{1}{\beta_0}+\frac{(3-3\lambda+\lambda^2)(8+3\lambda-\lambda^2)}{2\lambda(1-\lambda)(2-\lambda)^2}e^5
\left(-\frac{1}{e^{\beta_0}} - \int_{\beta_0}^\infty
\left(\frac{1}{\beta}\cdot \frac{{\rm d}}{{\rm d}
  \beta}\frac{\beta}{e^\beta} \right){\rm d}\beta \right)\\
&=&\frac{1}{\beta_0}+\frac{(3-3\lambda+\lambda^2)(8+3\lambda-\lambda^2)}{2\lambda(1-\lambda)(2-\lambda)^2}e^5
\left(-\frac{1}{e^{\beta_0}} + \int_{\beta_0}^\infty
\left(\frac{1}{\beta}\cdot \frac{\beta-1}{e^\beta} \right){\rm d}\beta
\right)\\
&=&\frac{1}{\beta_0}+\frac{(3-3\lambda+\lambda^2)(8+3\lambda-\lambda^2)}{2\lambda(1-\lambda)(2-\lambda)^2}e^5
\left(-\frac{1}{e^{\beta_0}} + \int_{\beta_0}^\infty
\frac{1}{e^\beta}-\frac{1}{\beta e^\beta}{\rm d}\beta
\right)\\
&=&\frac{1}{\beta_0}-\frac{(3-3\lambda+\lambda^2)(8+3\lambda-\lambda^2)}{2\lambda(1-\lambda)(2-\lambda)^2}e^5
\int_{\beta_0}^\infty \frac{1}{\beta e^\beta}{\rm d}\beta\\
&=&\frac{1}{\beta_0}-\frac{(3-3\lambda+\lambda^2)(8+3\lambda-\lambda^2)}{2\lambda(1-\lambda)(2-\lambda)^2}e^5
E_1(\beta_0),
\end{eqnarray*}
where $E_1(x)$ is the exponential integral
$$ {\rm E}_1(x)=\int_x^\infty \frac{1}{t e^t} {\rm d}t .$$
For every $x>0$, this exponential integral can be approximated by the
initial terms of the convergent series
$$ {\rm E}_1(x)=-\gamma-\ln x -\sum_{k=1}^\infty \frac{(-1)^k x^{k}}{k\cdot k!}, $$
where $\gamma=0.57721566\ldots$ is Euler's constant; see~\cite{AS64}.
With the choice of $\beta_0=9.955$ and $\lambda =0.452$, we obtain
$\rho \geq 0.09121$.

Taking into account Lemma~\ref{lem:greedy}, we summarize our main
result in the following theorem.
\begin{theorem}\label{thm:1}
For any finite point set $S\subset U$, $(0,0)\in S$, the algorithm {\sc
  TilePacking} chooses a set of rectangles of total area
$\rho \geq 0.09121$. Consequently, the same guarantee holds for the
algorithm {\sc GreedyPacking}.
\end{theorem}

\subsection{Runtime analysis\label{ssec:runtikme}}

It is not difficult to show that {\sc TilePacking} can be implemented
in $O(n\log n)$ time and $O(n)$ space in the RAM model of
computation. The input is a set $S$ of $n$ points in the unit square
$U=[0,1]^2$. Clearly, $S$ can be sorted in $O(n\log n)$ time in
non-increasing order of the sum of coordinates. Assume that the points
are labeled $s_1,s_2,\ldots , s_n$ in this order.

We compute the tiles sequentially in $n$ steps. Let $P_i=U\setminus
\bigcup_{j<i} t_j$ be the staircase polygon left from $U$ after
deleting the first $i-1$ tiles. We maintain the $x$- and $y$-coordinates
of the vertices of $P_i$, respectively, in two binary search trees.
In step $i$, we compute the tile $t_i$ by shooting a vertical
(resp., horizontal) ray from $s_i$ until it hits the boundary
of $P_i$. The point hit by an axis-parallel ray can be
found with a simple binary search in $O(\log n)$ time. Once the sides
$a_i$ and $b_i$ have been determined, we insert the $x$- and $y$-coordinates
of $s_i$ into the search trees, and delete the points that are not
vertices of $P_i$ subsequently. Each point in $S$ is inserted and deleted
at most once, so the search trees can be maintained in $O(n \log n)$
total time.

Recall that the generated tiles are staircase polygons, and their reflex
vertices are points in $S$. Since every point in $S$ is a reflex
vertex in at most one tile, the total complexity of the $n$ tiles is
$O(n)$. A tile with $k$ reflex vertices contains exactly $k+1$
maximal axis-aligned rectangles, and one with the largest area can be selected in
$O(k)$ time. Altogether, we can pick an axis-aligned rectangle of
maximum area from each of the $n$ tiles in $O(n)$ total time.
Hence {\sc TilePacking} runs in $O(n\log n)$ time.

\section{Conclusion}

We have shown that in the 1-round rectangle packing game, no matter
how Alice chooses a finite set of points in $S$, where $(0,0)\in S$, Bob can
always construct a rectangle packing with rectangles anchored at the
points in $S$ that cover at least a constant area. 
Allen Freedman~\cite{Tu69} asked whether Bob can always cover at least
1/2 of the unit square. Bill Pulleyblank and Peter Winkler 
conjectured that this is true. 
While we cannot confirm this at the moment, we believe that the
performance of our {\sc GreedyPacking} and {\sc TilePacking} algorithms
is significantly better than what we proved here.

We suspect that the problem of finding the rectangles with maximum total area
anchored at the given points is NP-hard, but this remains to be shown.
Our algorithms certainly achieve a constant approximation ratio $0.09121$.
No efficient exact algorithm or good approximation was previously known.

\paragraph{Special cases and variants.}
It is easily seen that the conjecture holds for ``permutation point sets'', 
namely integer $n$-element point sets from the 
$\{0,1,\ldots,n-1\} \times \{0,1,\ldots,n-1\}$ grid, with exactly 
one grid point in each row and column, and containing $(0,0)$, as required.
The unit square is now $[0,n] \times [0,n]$. 
This is in fact the only family of sets for which we could verify the 
conjecture. Indeed, for each point, say $(i,j)$, select the rectangle
$[i,n] \times [j,j+1]$; then the average width of the chosen
rectangles is $(n+1)/2$, each rectangle has unit height, and so 
the corresponding covered area ratio is
$\frac{1}{2}+\frac{1}{2n}$ for each of the $(n-1)!$ point sets.  
If the anchoring condition is relaxed so that the anchor point
of a rectangle can be either of its leftmost two vertices, then it is
easy to cover an area of at least $1/2$ as well. 

\paragraph{Higher-dimensional version.}
The $d$-dimensional generalization of the 1-round rectangle packing game
is also very interesting, and almost nothing is known about it.
If $S$ is a set of equally spaced points along the main diagonal of a
$d$-dimensional unit cube $U=[0,1]^d$, where $(0,\dots,0) \in S$,
then the total volume covered by any anchored $d$-dimensional
axis-parallel rectangle packing is roughly $1/d$.
Our {\sc GreedyPacking} and {\sc TilePacking} algorithms readily
generalize to $d$ dimensions, but the performance analysis does not
seem to be easily extendible. In particular, the definition of
$\beta$-tiles extends to arbitrary dimensions. Lemma~\ref{lem:sides}
also carries over (\ie, the volume of a $\beta$-tile is exponentially
smaller than the volume of its bounding box); and there are large
empty convex polytopes along the edges of a $\beta$-tile 
(in $d$-space, these are the $d$ edges incident to the anchor point)
similarly to the empty triangles $\Delta_i$ and $\Gamma_i$ in the plane. 
However, it is not clear what could be the analogues of the trapezoids
$A_i$ in higher dimensions and whether any charging scheme could be set up.

\paragraph{Multi-round versions.}
A natural generalization of the problem is the multi-round rectangle
packing game. One can consider two versions, depending on whether the number
of rounds is known in advance. In the {\em $n$-round rectangle packing
  game}, both Alice and Bob know the number of rounds. In round $i$,
first Alice places a point $s_i\in [0,1]^2$ somewhere outside of Bob's
rectangles, and then Bob chooses an axis-aligned rectangle
$r_i\subseteq [0,1]^2$ with lower left corner at $s_i$ and
interior-disjoint from his previous rectangles. Alice has to choose
the origin in one of the $n$ rounds. In the {\em unlimited rectangle
  packing game}, the number of rounds (or points) is not known in
advance. Each round goes exactly as in the $n$-round version, but the
game terminates when Alice decides to put a point at the origin and
Bob chooses his last rectangle incident to the origin.

For both versions of the multi-round game, Bob could employ a greedy
strategy: for each point $s_i$, let $r_i$ be an axis-aligned rectangle
of maximum area with lower left corner at $s_i$ that is interior-disjoint
from all previous rectangles. However, our analysis does not
extend to these versions of the game. In fact, we can show that the
greedy strategy 
cannot guarantee any constant area for Bob.
Whether Bob can secure a constant fraction of the area by other means 
in any of the multi-round versions of the game remains open. 

\begin{theorem}\label{thm:multiround}
In both multi-round versions of the rectangle packing game, Bob
cannot always cover $\Omega(1)$ area with a greedy strategy.
\end{theorem}
\begin{proof}
We show that for every $\eps>0$, Alice can construct a finite sequence
of $n$ points $s_1,\ldots , s_n$, such that Bob can cover at most
$\eps$ area with a greedy strategy.
Essentially, Alice can force Bob to choose a rectangle from
a $\frac{\eps}{2}$-tile (using at most $\frac{\eps}{2}$ of the tile's area)
and then fence off the remainder of the tile so that it cannot be covered
later. Alice can make sure that the total area of these $\frac{\eps}{2}$-tiles
is arbitrarily close to 1, say $1-\frac{\eps}{2}$. Then Bob can cover
at most $\frac{\eps}{2} + (1-\frac{\eps}{2})\cdot \frac{\eps}{2}<\eps$ area.
We proceed with the details.

We say that a staircase polygon $P$ is a $\beta$-staircase, for some
$\beta\geq 1$, if the area of every axis-aligned rectangle contained
in $P$ is at most $\area(P)/\beta$. By Lemma~\ref{lem:sides}, for
every $\beta \geq 1$, $h>0$, and $w>0$, one can construct a $\beta$-staircase
of height $h$ and width $w$ whose area is roughly $\beta e^{1-\beta} hw$.

\begin{figure}[htbp]
\centering
\includegraphics[width=.95\textwidth]{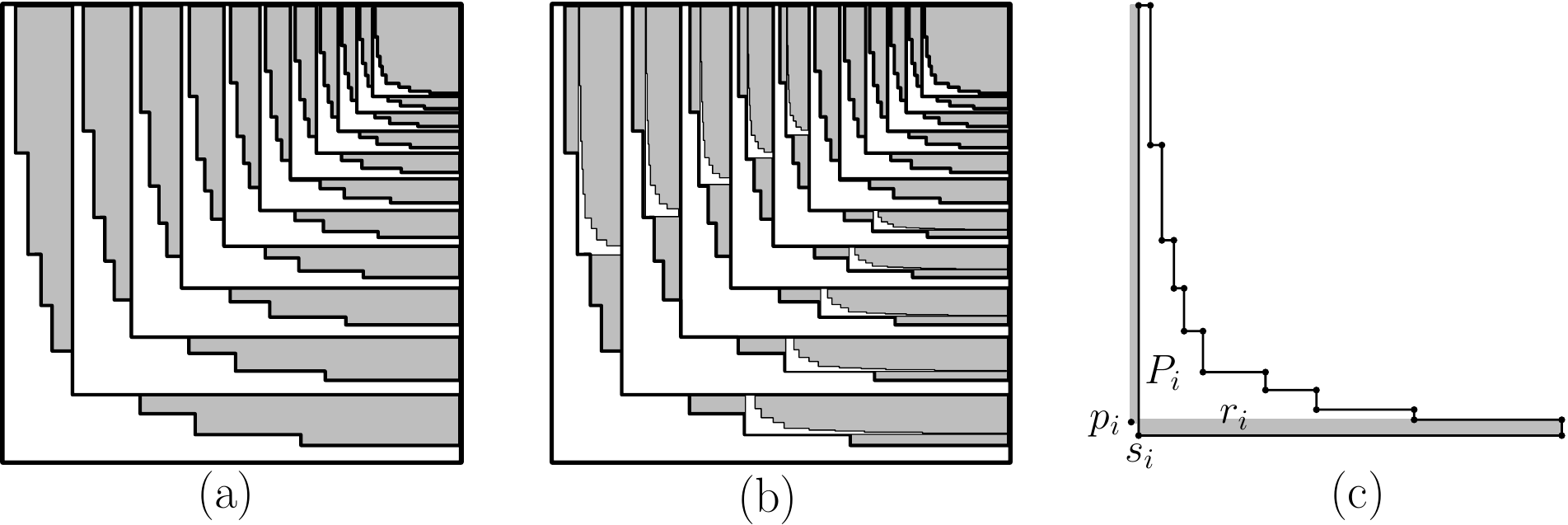}
\caption{(a) An initial $\frac{2}{\eps}$-staircase packing.
(b) In any empty region (shaded gray), Alice can place an additional
  $\frac{2}{\eps}$-staircase.
(c) The placement of an extra point $p_i$ near the left side of
the maximal rectangle $r_i\subset P_i$.}
\label{fig:6}
\end{figure}

Alice first computes a packing of the unit square $U=[0,1]^2$ with
$\frac{2}{\eps}$-staircases, by successively choosing interior-disjoint
$\frac{2}{\eps}$-staircases of smaller and smaller sizes until
their total area is at least $1-\frac{\eps}{2}$ (Fig.~\ref{fig:6}(a,b)).
Specifically, in the current step, given an axis aligned rectangle $R$
of height $h$ and width $w$, a $\frac{2}{\eps}$ staircase polygon $P$
with the same height and width is anchored at the lower left corner of
$R$, and the remaining space $R \setminus P$ is partitioned into
vertical or horizontal sectors to be processed (see Section~\ref{sec:analysis}).

Then she slightly shrinks these staircases, to make them
pairwise disjoint, and perturbs them to ensure that each
$\frac{2}{\eps}$-staircase $P_i$ contains a {\em unique} axis-aligned
rectangle $r_i$ of maximum area, and $r_i$ has the same width as $P_i$.
The point set $S$ contains, for every $\frac{2}{\eps}$-staircase
in this packing, all vertices of $P_i$, including the lower left
corner $s_i$. In addition,  for every $\frac{2}{\eps}$-staircase
whose lower left corner is not on the left side of $U$, $S$ also
contains a point $p_i$ very close to the left side of rectangle $r_i$,
as shown in Fig.~\ref{fig:6}(c).

It remains to determine the order in which Alice reveals the points to
Bob. The points associated with each $\frac{2}{\eps}$-staircase $P_i$
are revealed in a contiguous sequence such that the last two points
in each sequence are the lower left corner $s_i$ followed by the
extra point $p_i$. For the lower left corner $s_i$, Bob has to choose
the unique rectangle $r_i$ of maximum area in $P_i$, which is adjacent to the
lower side of $P_i$. For the extra point $p_i$, Bob has to choose a tall
rectangle of negligible area, which covers the left side of $P_i$.
These two rectangles guarantee that no subsequent rectangle can
cover any additional part of $P_i$, while
$\area(r_i) \leq \frac{\eps}{2} \cdot \area(P_i)$.

To determine the order of sequences of points associated with the
staircases, we define a partial order over the staircase polygons.
Note that in the initial staircase packing, each tip of every $P_i$
is adjacent to the left or lower side of another staircase, or the
right or upper side of $U$.
This defines a partial order: let $P_i\prec P_j$, if the tip of $P_j$ is
adjacent to the left or lower side of $P_i$; or if the lower left corner
of $P_i$ is a reflex vertex of $P_j$. Order the staircases in
any linear extension of this partial order. This ensures that Bob cannot
choose a rectangle intersecting the interior of $P_i$ before Alice
reveals the lower left corner $s_i$.
\end{proof}

\paragraph{Best versus worst greedy strategy.}
For a finite set $S\subset [0,1]^2$, and a permutation (ordering)
$\pi$ of $S$, we can select anchored rectangles greedily (ties are
broken arbitrarily) in the order prescribed by $\pi$. One can ask
which permutation gives the best or the worst performance
for a greedy strategy.
Our main result, Theorem~\ref{thm:1}, says that for every $n$-element
point set $S$, $(0,0)\in S$, we can find in $O(n\log n)$ time a
permutation $\pi$ for which the greedy strategy covers
$\Theta(1)$ area. In the worst case, greedy covers only $o(1)$ area by
Theorem~\ref{thm:multiround}.
In the best case, however, we will show (Lemma~\ref{lem:opt})
that greedy is always optimal for some permutation $\pi$. We say that
an anchored rectangle packing is {\em Pareto optimal} if each
rectangle $r(s)$ has maximum area assuming that all other rectangles
are fixed. It is clear that every optimal solution is Pareto optimal.

In particular, each rectangle in an optimal solution is bounded by the
two rays (going up and to the right) from its anchor point, and
two other such rays (from other points) that limit it from the right and from the top.
This immediately implies the existence of an exact algorithm for
the optimization problem running in exponential time, based on brute
force enumeration. The next lemma also shows that the greedy algorithm
and brute force enumeration of permutations yields
yet another exact algorithm for the optimization problem.

\begin{lemma}\label{lem:opt}
For every finite point set $S\subset [0,1]^2$ and every Pareto optimal
anchored packing $R=\{r(s):s\in S\}$, there is a permutation $\pi$
for which a greedy algorithm (with some tie breaking) computes $R$.
\end{lemma}
\begin{proof}
Let $S\subset [0,1]^2$ be a finite set, which may not contain the origin.
Since every $r(s)$ is Pareto optimal, it is a greedy choice assuming
that $s$ is the last point in the order $\pi$. Suppose that $s\in S$ is the
last point in a permutation $\pi$. If the smaller problem $S\setminus \{s\}$
with $R\setminus \{r(s)\}$ is not Pareto optimal, then there is a point $s'\in S$
for which we could choose a larger rectangle anchored at $s'$, which
intersects $r(s)$ only. 
In this case, either $s$ dominates $s'$ or the ray shot from $s'$ vertically
up (resp., horizontally right) hits the lower (resp., left) side of rectangle $r(s)$.
This motivates the definition of a binary relation over $S$, which is
an extension of the dominance order. Let $s'\prec s$ if either $s$
dominates $s'$ or an axis-aligned ray shot from $s'$ hits the lower or
left side of rectangle $r(s)$. It is not difficult to see that this is
a partial order over $S$. 
If $s\in S$ is a minimal element in the poset $(S,\prec)$, then the rectangles
in $R\setminus \{r(s)\}$ are still Pareto optimal for the anchors $S\setminus \{s\}$.
Now let $\pi$ be the reverse order of any linear extension of this partial order.
\end{proof}

We have shown (our main result) that for any set of $n$ points in the
unit square $U=[0,1]^2$, one can find a set of disjoint empty
rectangles anchored at the given points and covering more than $9\%$
of $U$. The same conclusion holds for points in any axis-aligned rectangle $V$
instead of $U$, since it is straightforward to use an affine transformation
to map the input into the unit square. Concerning the bound obtained,
a sizable gap to the conjectured $50\%$ remains. While certainly small
adjustments in our proof can lead to improvements in the bound,
obtaining substantial improvements probably requires new ideas.

\old{
\paragraph{Open problems.} We conclude with some obvious questions:
\begin{enumerate} \itemsep 1pt
\item Do the rectangles constructed by {\sc TilePacking} always cover
an area of at least $1/2$ in the plane?
\item What is the answer to the question in $3$-space?
Do the rectangular boxes constructed by {\sc TilePacking} always cover
a constant volume (perhaps $1/3$)?
\end{enumerate} 
} 

\paragraph{Acknowledgement.}
The authors thank Richard Guy for tracing back the origins of this problem 
and Jan Kyn\v{c}l for comments and remarks.

\end{document}